\documentclass[preprint,12pt,authoryear]{elsarticle}

\usepackage{amsmath,amssymb,amsthm,mathtools}
\usepackage{hyperref}

\newtheorem{theorem}{Theorem}
\newtheorem{lemma}{Lemma}
\newtheorem{corollary}{Corollary}
\newtheorem{proposition}{Proposition}
\newtheorem{definition}{Definition}
\newtheorem{example}{Example}
\newtheorem{remark}{Remark}

\newcommand{\E}{\mathbb{E}}
\newcommand{\Prob}{\mathbb{P}}
\newcommand{\R}{\mathbb{R}}
\newcommand{\C}{\mathbb{C}}

\newcommand{\Law}{\mathcal{L}}
\newcommand{\dd}{\,\mathrm{d}}
\newcommand{\Repart}{\operatorname{Re}}
\newcommand{\scle}{\preceq_{\mathrm{sc}}}

\numberwithin{equation}{section}

\journal{Statistics \& Probability Letters}

\begin{document}

\begin{frontmatter}

\title{Stable size-biasing and the positive scale-mixture order of generalized Gaussian laws}

\author[addr1]{Domingos S. P. Salazar\corref{cor1}}
\address[addr1]{Unidade de Educa\c c\~ao a Dist\^ancia e Tecnologia,
Universidade Federal Rural de Pernambuco,
52171-900 Recife, Pernambuco, Brazil}
\cortext[cor1]{Corresponding author.}

\begin{abstract}
Let $X_r\sim N_r(0,1)$ be the centered unit-scale generalized Gaussian random variable with density proportional to $\exp(-|x|^r/2)$. We prove that, for $p,q>0$, there exists a strictly positive random variable $V$, independent of $X_q$, such that
\[
X_p\stackrel{d}{=}VX_q
\]
if and only if $p\le q$. Moreover, the law of $V$ is unique. For $p<q$, put $a=1/p$, $b=1/q$, and $\alpha=b/a=p/q$. If $S_\alpha$ is positive $\alpha$-stable, $\E e^{-uS_\alpha}=e^{-u^\alpha}$, set $W_0=S_\alpha^{-b}$, let $W$ be the $W_0$-size-biased version of $W_0$, and define
\[
V_{p,q}=2^{a-b}W.
\]
Then $X_p\stackrel{d}{=}V_{p,q}X_q$. For $p>q$, the necessary Mellin quotient for $t\mapsto \E e^{it\log V}$ is unbounded by Stirling's formula, and hence cannot be a characteristic function. The factor laws form a multiplicative cocycle, $V_{p,r}\stackrel{d}{=}V_{p,q}V_{q,r}$, for $p\le q\le r$, where the factors on the right-hand side are independent copies. Thus the Mellin quotient isolated by Dytso, Bustin, Poor and Shamai is realized constructively throughout the $p<q$ branch. In particular, $\Phi_{p,q}$ is positive definite exactly in the range $p\le q$, and the inverse Fourier--Mellin candidate density in the remaining $p<q$ branch is a genuine nonnegative probability density. The known Gaussian-base and bounded-parameter product cases are recovered as parts of a single positive scale-mixture classification.
\end{abstract}


\end{frontmatter}

\section{Introduction}

For $r>0$, let $X_r\sim N_r(0,1)$ be the centered generalized Gaussian random variable with density
\begin{equation}
 f_r(x)=\frac{r}{2^{1+1/r}\Gamma(1/r)}\exp\left(-\frac{|x|^r}{2}\right),\qquad x\in\R .
\label{eq:density}
\end{equation}
This is the normalization used by \citet{Dytso2018}. Throughout, $N_r(0,1)$ refers to this scale normalization; it is not a variance-one normalization except when $r=2$, since
\[
\E X_r^2=2^{2/r}\frac{\Gamma(3/r)}{\Gamma(1/r)}.
\]
The parameter $r$ controls tail decay and concentration: smaller $r$ gives heavier tails, while larger $r$ gives lighter tails and stronger concentration. The family includes the Laplace law at $r=1$ and the standard Gaussian law at $r=2$, under the present scaling; see also \citet{Subbotin1923} and \citet{Nadarajah2005} for related historical and statistical treatments of generalized normal laws.

Scale-mixture representations for exponential-power and generalized Gaussian laws have a substantial prior literature. In the Gaussian-base case, \citet{West1987} showed that an exponential-power family is contained in the class of scale mixtures of normals and obtained the corresponding mixing distributions. Multivariate Gaussian-mixture branches for exponential-power distributions were developed by \citet{GomezSanchezManzano2008}, with Bayesian applications. More recently, \citet{Korolev2020} surveyed and extended univariate and multivariate mixture representations for exponential-power laws and related distributions, emphasizing product representations with independent multipliers. These works cover important Gaussian-base and multivariate branches. They do not, however, decide the full pairwise question of when $X_p$ is a positive independent scale mixture of $X_q$ for arbitrary $p,q>0$, nor do they identify the unique inter-shape factor for every admissible pair.

\citet{Dytso2018} studied stochastic orders, Mellin transforms, product decompositions, characteristic functions, infinite divisibility and self-decomposability for this family. In their notation, Proposition~5 proves
\[
X_Q\stackrel{d}{=}U_{P,Q}X_{2Q/P},\qquad 0<Q\le P\le 2,
\]
with $U_{P,Q}>0$ independent of $X_{2Q/P}$. Setting $Q=p$ and $P=2p/q$, the conditions $0<Q\le P\le 2$ are equivalent to $0<p\le q\le 2$. Thus Proposition~5 supplies
\[
X_p\stackrel{d}{=}UX_q
\]
on the product subregion $0<p\le q\le 2$.

Their concluding discussion asks whether the more general product representation
\begin{equation}
X_p\stackrel{d}{=}VX_q,\qquad V>0,\quad V\perp X_q,
\label{eq:product-question}
\end{equation}
can hold. Taking absolute values and using Mellin transforms shows that any such factor must satisfy
\begin{equation}
\E V^{it}=\Phi_{p,q}(t):=\frac{2^{it/p}\Gamma((1+it)/p)\Gamma(1/q)}{2^{it/q}\Gamma((1+it)/q)\Gamma(1/p)} .
\label{eq:Phi}
\end{equation}
Thus the product problem is equivalent to a positive-definiteness problem for $\Phi_{p,q}$. Necessity follows from \eqref{eq:Phi}. Conversely, if $\Phi_{p,q}$ is the characteristic function of a real random variable $Y$, then, taking $V=e^Y$ independent of $X_q$, the logarithms of $|VX_q|$ and $|X_p|$ have the same characteristic function. Since the generalized Gaussian laws are symmetric and have no atom at zero, attaching the independent Rademacher sign of $X_q$ gives $VX_q\stackrel{d}{=}X_p$. Proposition~15 of \citet{Dytso2018} proves that $\Phi_{p,q}$ is not a characteristic function when $p>q$. When $p<q$, it proves integrability of $\Phi_{p,q}$ and gives the corresponding inverse Fourier--Mellin candidate density, conditional on that candidate being a genuine probability density. The unresolved step is therefore the positive-definiteness of $\Phi_{p,q}$ throughout the $p<q$ branch not already supplied by the product representation in their Proposition~5. The present note proves this step by constructing the corresponding strictly positive factor explicitly.

The contribution of this note is intentionally precise. We prove the full positive scale-mixture order for all $p,q>0$:
\[
\exists\,V>0,\ V\perp X_q,\ X_p\stackrel{d}{=}VX_q\quad\Longleftrightarrow\quad p\le q.
\]
For $p<q$, the Mellin quotient is realized by an explicit strictly positive random variable built from a positive stable law and ordinary size-biasing. For $p>q$, the same quotient is necessarily the characteristic function of $\log V$, but its modulus is unbounded, which is impossible for a characteristic function. Relative to the previous literature, the new content is not another Gaussian scale-mixture formula; it is the exact pairwise order, the explicit positive-stable size-biased factor, uniqueness of the factor law, and the extension beyond both the Gaussian-base branch and the earlier product subregion $0<p\le q\le 2$.

This perspective is useful because the generalized Gaussian parameter $r$ is used precisely to tune tail behavior and concentration. If $\nu_{p,q}$ denotes the law of the factor, then, for every bounded measurable $h$,
\[
\E h(X_p)=\int_0^\infty \E h(vX_q)\,\nu_{p,q}(\dd v),\qquad p\le q.
\]
The weighted form of the construction also gives an exact importance identity for moving between shape parameters. Thus the factor supplies a canonical transport kernel between generalized Gaussian shapes, not only a positivity certificate for an inverse Mellin formula.

\paragraph{Contents.} Section~\ref{sec:prelim} records the Mellin transform of $|X_r|$, the negative moments of a positive stable variable, and the size-biased stable factor. Section~\ref{sec:theorem} proves the classification theorem. Section~\ref{sec:consequences} gives the transport-kernel consequences and benchmark cases. Section~\ref{sec:conclusion} summarizes the result and possible extensions.

\section{Preliminaries}
\label{sec:prelim}

We start with the absolute Mellin transform of the generalized Gaussian law.

\begin{lemma}[Absolute Mellin transform]
Let $X_r\sim N_r(0,1)$, with density \eqref{eq:density}. Then, for $\Repart s>-1$,
\begin{equation}
\E |X_r|^s=2^{s/r}\frac{\Gamma((s+1)/r)}{\Gamma(1/r)}.
\label{eq:absolute-mellin}
\end{equation}
\end{lemma}

\begin{proof}
By symmetry and the normalization in \eqref{eq:density},
\[
\E |X_r|^s=\frac{r}{2^{1/r}\Gamma(1/r)}\int_0^\infty x^s e^{-x^r/2}\dd x .
\]
With $y=x^r/2$, one has $x=(2y)^{1/r}$ and $\dd x=(2^{1/r}/r)y^{1/r-1}\dd y$. Hence
\[
\E |X_r|^s=\frac{2^{s/r}}{\Gamma(1/r)}\int_0^\infty y^{(s+1)/r-1}e^{-y}\dd y
=2^{s/r}\frac{\Gamma((s+1)/r)}{\Gamma(1/r)},
\]
provided $\Repart s>-1$.
\end{proof}

We next fix the stable-law normalization used throughout.

\begin{definition}[Positive $\alpha$-stable law]
Let $0<\alpha<1$. A positive $\alpha$-stable random variable $S_\alpha$ is the positive random variable whose Laplace transform is
\begin{equation}
\E e^{-uS_\alpha}=e^{-u^\alpha},\qquad u\ge0.
\label{eq:stable-laplace}
\end{equation}
Equivalently, $S_\alpha$ is the time-one marginal of the stable subordinator with Laplace exponent $u^\alpha$. This is the standard one-sided stable normalization used in classical treatments of stable laws \citep{Feller1971,Zolotarev1986,Nolan2020}.
\end{definition}

With this normalization, we recall the standard negative-moment formula.

\begin{lemma}[Negative moments of a positive stable law]
Let $0<\alpha<1$, and let $S_\alpha$ be as in Definition~1. Then, for $z\in\C$ with $\Repart z>0$,
\begin{equation}
\E S_\alpha^{-z}=\frac{\Gamma(z/\alpha)}{\alpha\Gamma(z)}.
\label{eq:stable-negative-moment}
\end{equation}
\end{lemma}

\begin{proof}
For $x>0$ and $\Repart z>0$,
\[
x^{-z}=\frac{1}{\Gamma(z)}\int_0^\infty u^{z-1}e^{-ux}\dd u .
\]
Moreover,
\[
\int_0^\infty u^{\Repart z-1}\E e^{-uS_\alpha}\dd u
=\int_0^\infty u^{\Repart z-1}e^{-u^\alpha}\dd u<\infty .
\]
Hence Fubini's theorem gives
\[
\E S_\alpha^{-z}=\frac{1}{\Gamma(z)}\int_0^\infty u^{z-1}e^{-u^\alpha}\dd u .
\]
The substitution $v=u^\alpha$ gives
\[
\E S_\alpha^{-z}=\frac{1}{\alpha\Gamma(z)}\int_0^\infty v^{z/\alpha-1}e^{-v}\dd v
=\frac{\Gamma(z/\alpha)}{\alpha\Gamma(z)}.
\]
\end{proof}

We shall use ordinary size-biasing.

\begin{definition}[Size-biased law]
Let $Y\ge0$ satisfy $0<\E Y<\infty$. The $Y$-size-biased version $Y^*$ of $Y$ is the positive random variable whose law is defined by
\[
\Prob(Y^*\in A)=\frac{\E[Y\mathbf{1}_{\{Y\in A\}}]}{\E Y},
\]
for Borel sets $A\subseteq[0,\infty)$. Equivalently,
\[
\E h(Y^*)=\frac{\E[Yh(Y)]}{\E Y}
\]
for all bounded measurable $h$.
\end{definition}

The factor used in the proof is the following inverse-stable size-biased variable.

\begin{lemma}[Inverse-stable size-biased factor]
\label{lem:inverse-stable-size-bias}
Let $a>b>0$, set $\alpha=b/a\in(0,1)$, and let $S_\alpha$ be as in Lemma~2. Put
\[
W_0=S_\alpha^{-b},
\]
and let $W$ be the $W_0$-size-biased version of $W_0$. Then, for $\Repart s>-1$,
\begin{equation}
\E W^s=\frac{\Gamma(a(s+1))\Gamma(b)}{\Gamma(b(s+1))\Gamma(a)}.
\label{eq:W-mellin}
\end{equation}
Consequently, if
\[
V=2^{a-b}W,
\]
then
\begin{equation}
\E V^s=2^{(a-b)s}\frac{\Gamma(a(s+1))\Gamma(b)}{\Gamma(b(s+1))\Gamma(a)},\qquad \Repart s>-1.
\label{eq:V-mellin}
\end{equation}
\end{lemma}

\begin{proof}
For $\Repart z>0$, Lemma~2 gives
\[
\E W_0^z=\E S_\alpha^{-bz}=\frac{\Gamma(bz/\alpha)}{\alpha\Gamma(bz)}
=\frac{\Gamma(az)}{\alpha\Gamma(bz)},
\]
since $b/\alpha=a$. In particular,
\[
\E W_0=\frac{\Gamma(a)}{\alpha\Gamma(b)}<\infty .
\]
Let $\sigma=\Repart s>-1$. Since the law of $W$ is
\[
\Prob(W\in\dd x)=\frac{x\Prob(W_0\in\dd x)}{\E W_0},
\]
we have
\[
\E |W^s|=\frac{\E W_0^{\sigma+1}}{\E W_0}<\infty .
\]
Therefore, using $x^s=e^{s\log x}$ for $x>0$,
\[
\E W^s=\frac{1}{\E W_0}\int_0^\infty x^{s+1}\Prob(W_0\in\dd x)
=\frac{\E W_0^{s+1}}{\E W_0}
=\frac{\Gamma(a(s+1))\Gamma(b)}{\Gamma(b(s+1))\Gamma(a)}.
\]
Multiplication by $2^{a-b}$ gives \eqref{eq:V-mellin}.
\end{proof}

\begin{definition}[Positive scale-mixture relation]
For probability laws $\mu$ and $\nu$ on $\R$, write
\[
\mu\scle\nu
\]
if there exist $Z\sim\nu$ and a random variable $V>0$ almost surely, independent of $Z$, such that $VZ\sim\mu$. For random variables $Y$ and $Z$, write $Y\scle Z$ when $\Law(Y)\scle\Law(Z)$. Here and below, $\Law(X)$ denotes the law, or distribution, of a random variable $X$.
\end{definition}

The definition requires $V>0$. For the generalized Gaussian laws considered below, the same relation would be obtained if one allowed $V\ge0$: indeed, if $X_p\stackrel{d}{=}VX_q$, $V\ge0$, and $V\perp X_q$, then
\[
\Prob(VX_q=0)\ge \Prob(V=0),
\]
whereas $X_p$ has no atom at zero. Hence $\Prob(V=0)=0$. At the level of laws, $\scle$ is a scale-mixture preorder: reflexivity is given by $V=1$, and transitivity follows by taking fresh independent copies of the two scaling factors. Theorem~\ref{thm:main} shows that, on the centered unit-scale generalized Gaussian family, this preorder is exactly the usual order of shape parameters.

\begin{lemma}[Mellin quotient and uniqueness]
\label{lem:quotient-uniqueness}
Let $p,q>0$. If $V>0$, $V\perp X_q$, and $X_p\stackrel{d}{=}VX_q$, then, for every $t\in\R$,
\begin{equation}
\E V^{it}=\frac{2^{it/p}\Gamma((1+it)/p)\Gamma(1/q)}{2^{it/q}\Gamma((1+it)/q)\Gamma(1/p)}=:\Phi_{p,q}(t).
\label{eq:quotient}
\end{equation}
Consequently, whenever a strictly positive independent factor exists, its law is unique.
\end{lemma}

\begin{proof}
Taking absolute values in $X_p\stackrel{d}{=}VX_q$ and using independence gives
\[
\E |X_p|^{it}=\E V^{it}\E |X_q|^{it}.
\]
By Lemma~1,
\[
\E |X_q|^{it}=2^{it/q}\frac{\Gamma((1+it)/q)}{\Gamma(1/q)}.
\]
This quantity is nonzero because the gamma function has no zeros. Division gives \eqref{eq:quotient}. Since $t\mapsto \E V^{it}$ is the characteristic function of $\log V$, and characteristic functions determine probability laws \citep{Lukacs1970}, the law of $\log V$, hence the law of $V$, is unique.
\end{proof}

\section{The factorization theorem}
\label{sec:theorem}

\begin{theorem}[Generalized Gaussian product factorization]
\label{thm:main}
Let $p,q>0$, and let $X_p\sim N_p(0,1)$ and $X_q\sim N_q(0,1)$. Then
\[
X_p\scle X_q\quad\Longleftrightarrow\quad p\le q.
\]
Equivalently, there exists a strictly positive random variable $V$, independent of $X_q$, such that
\[
X_p\stackrel{d}{=}VX_q
\]
if and only if $p\le q$. When such $V$ exists, its law is unique.

More precisely:
\begin{enumerate}
\item If $p=q$, one may take $V=1$.
\item If $p<q$, put $a=1/p$, $b=1/q$, $\alpha=b/a=p/q$, let $S_\alpha$ be positive $\alpha$-stable, put $W_0=S_\alpha^{-b}$, let $W$ be the $W_0$-size-biased version of $W_0$, chosen independent of $X_q$, and define
\[
V=2^{a-b}W.
\]
Then $V>0$ almost surely, $V\perp X_q$, and $VX_q\stackrel{d}{=}X_p$.
\item If $p>q$, no such strictly positive independent factor exists.
\end{enumerate}
\end{theorem}

\begin{proof}
The case $p=q$ is immediate.

Assume first that $p<q$. Put
\[
a=\frac{1}{p},\qquad b=\frac{1}{q}.
\]
Then $a>b>0$. Let $V$ be the strictly positive random variable defined in Lemma~3, independent of $X_q$. For $s\in\C$ with $\sigma=\Repart s>-1$, the relevant complex moments are absolutely integrable, since
\[
\E[|V^s||X_q|^s]=\E V^\sigma\E |X_q|^\sigma<\infty .
\]
For $\Repart s>-1$, Lemma~3 and Lemma~1 give
\[
\E V^s=2^{(a-b)s}\frac{\Gamma(a(s+1))\Gamma(b)}{\Gamma(b(s+1))\Gamma(a)}
\]
and
\[
\E |X_q|^s=2^{bs}\frac{\Gamma(b(s+1))}{\Gamma(b)}.
\]
Therefore,
\begin{align}
\E |VX_q|^s
&=\E V^s\E |X_q|^s \notag\\
&=2^{(a-b)s}\frac{\Gamma(a(s+1))\Gamma(b)}{\Gamma(b(s+1))\Gamma(a)}\cdot 2^{bs}\frac{\Gamma(b(s+1))}{\Gamma(b)} \notag\\
&=2^{as}\frac{\Gamma(a(s+1))}{\Gamma(a)}
=\E |X_p|^s .
\label{eq:abs-moment-match}
\end{align}
Taking $s=it$ in \eqref{eq:abs-moment-match} yields
\[
\E e^{it\log |VX_q|}=\E e^{it\log |X_p|},\qquad t\in\R.
\]
Since $|X_p|>0$ and $|VX_q|>0$ almost surely, this is equality of the characteristic functions of $\log |VX_q|$ and $\log |X_p|$. Therefore
\[
|VX_q|\stackrel{d}{=}|X_p|.
\]
The density \eqref{eq:density} is symmetric and has no atom at zero, so $X_r=\varepsilon_r|X_r|$, where $\varepsilon_r$ is a Rademacher random variable independent of $|X_r|$. Since $V>0$ and $V\perp X_q$, the sign of $VX_q$ is $\varepsilon_q$, and $\varepsilon_q$ is independent of $(V,|X_q|)$, hence independent of $|VX_q|$. Thus both variables are obtained by attaching an independent Rademacher sign to the same absolute-value law, and
\[
VX_q\stackrel{d}{=}X_p.
\]

It remains to prove nonexistence for $p>q$. Suppose, for contradiction, that $V>0$, $V\perp X_q$, and $X_p\stackrel{d}{=}VX_q$. By Lemma~\ref{lem:quotient-uniqueness}, $\Phi_{p,q}$ in \eqref{eq:quotient} would have to be the characteristic function of $\log V$, and therefore would have to be bounded in modulus by one.

We include the short Stirling argument for completeness; it is the same obstruction underlying Proposition~15 of \citet{Dytso2018}. The fixed-real-part gamma asymptotic
\[
|\Gamma(x+iy)|\sim \sqrt{2\pi}\,|y|^{x-1/2}e^{-\pi |y|/2},\qquad x>0,\\ |y|\to\infty,
\]
gives, as $|t|\to\infty$,
\[
\left|\Gamma\left(\frac{1+it}{r}\right)\right|
\sim C_r |t|^{1/r-1/2}\exp\left(-\frac{\pi |t|}{2r}\right),
\]
with $C_r>0$. Therefore,
\[
|\Phi_{p,q}(t)|\sim C_{p,q}|t|^{1/p-1/q}\exp\left[\frac{\pi |t|}{2}\left(\frac{1}{q}-\frac{1}{p}\right)\right],
\qquad |t|\to\infty,
\]
for some $C_{p,q}>0$. If $p>q$, then $(1/q-1/p)>0$, so $|\Phi_{p,q}(t)|\to\infty$. This contradicts the boundedness of characteristic functions. Hence no strictly positive independent $V$ can exist when $p>q$. The uniqueness assertion follows from Lemma~\ref{lem:quotient-uniqueness}.
\end{proof}

\begin{corollary}[Distributional multiplicative cocycle]
\label{cor:cocycle}
Let $0<p\le q\le r$, and let $V_{u,v}$ denote a positive random variable with the unique factor law such that $X_u\stackrel{d}{=}V_{u,v}X_v$. If $V_{p,q}$ and $V_{q,r}$ are chosen independent of each other and of $X_r$, then
\[
V_{p,r}\stackrel{d}{=}V_{p,q}V_{q,r}.
\]
\end{corollary}

\begin{proof}
Take $X_r$, $V_{q,r}$, and $V_{p,q}$ mutually independent, with the prescribed factor laws. Put
\[
X_q'=V_{q,r}X_r.
\]
By Theorem~\ref{thm:main}, $X_q'\stackrel{d}{=}X_q$, and $X_q'$ is independent of $V_{p,q}$. Applying the factorization from $q$ to $p$ to this fresh copy $X_q'$, we get
\[
V_{p,q}V_{q,r}X_r=V_{p,q}X_q'\stackrel{d}{=}X_p.
\]
Thus $V_{p,q}V_{q,r}$ is a strictly positive independent factor from $X_r$ to $X_p$. The law of such a factor is unique by Lemma~\ref{lem:quotient-uniqueness}.
\end{proof}

\begin{corollary}[Positive definiteness of the Mellin quotient]
\label{cor:positive-definiteness}
For $p,q>0$, the function
\[
\Phi_{p,q}(t)=\frac{2^{it/p}\Gamma((1+it)/p)\Gamma(1/q)}{2^{it/q}\Gamma((1+it)/q)\Gamma(1/p)}
\]
is a characteristic function if and only if $p\le q$. If $p<q$, then $\Phi_{p,q}$ is the characteristic function of
\[
\log V_{p,q}=\left(\frac{1}{p}-\frac{1}{q}\right)\log 2+\log W,
\qquad
W=\left(S_{p/q}^{-1/q}\right)^*,
\]
where the star denotes ordinary size-biasing of the positive random variable $S_{p/q}^{-1/q}$. If $p=q$, then $\Phi_{p,p}\equiv1$.
\end{corollary}

\begin{proof}
If $p\le q$, Theorem~\ref{thm:main} gives a strictly positive independent factor $V_{p,q}$, and Lemma~\ref{lem:quotient-uniqueness} identifies $\Phi_{p,q}$ with the characteristic function of $\log V_{p,q}$. If $p>q$, the Stirling estimate in the proof of Theorem~\ref{thm:main} shows that $|\Phi_{p,q}(t)|\to\infty$, which is impossible for a characteristic function.
\end{proof}

\begin{remark}[Relation to the inverse Fourier--Mellin candidate]
For $p<q$, the same Stirling estimate gives $\Phi_{p,q}\in L^1(\R)$. With the convention
\[
g_{p,q}(y)=\frac{1}{2\pi}\int_\R e^{-ity}\Phi_{p,q}(t)\dd t,
\qquad y\in\R,
\]
Fourier inversion gives a continuous density $g_{p,q}$ for $\log V_{p,q}$. Therefore $V_{p,q}$ has density
\[
f_{p,q}(v)=\frac{1}{v}g_{p,q}(\log v),\qquad v>0.
\]
This identifies the inverse Fourier--Mellin candidate of \citet{Dytso2018} with the density of the explicit positive factor $V_{p,q}$, and proves its nonnegativity.
\end{remark}

\section{Consequences and benchmark cases}
\label{sec:consequences}

\subsection{Moment determinacy of the inter-shape factor}
\label{subsec:factor-moment-determinacy}

The explicit stable-size-biased representation also settles the Stieltjes moment problem for the factor itself. This sharpens the Gaussian-base determinacy result of \citet[Proposition~8]{Dytso2018} from a one-sided sufficient condition to an exact two-parameter phase diagram.

\begin{definition}[Stieltjes moment determinacy]
Let $Y\ge0$ have finite moments $m_k=\E Y^k$, $k=0,1,2,\ldots$. The law of $Y$ is Stieltjes moment-determinate if it is the only probability law on $[0,\infty)$ with moment sequence $(m_k)_{k\ge0}$. Otherwise it is Stieltjes moment-indeterminate.
\end{definition}

\begin{theorem}[Sharp Stieltjes moment determinacy of the factor]
\label{thm:factor-moment-determinacy}
Let $0<p\le q$, and let $V_{p,q}$ be the unique positive factor in Theorem~\ref{thm:main}. Then $V_{p,q}$ is Stieltjes moment-determinate if and only if
\[
\frac1p-\frac1q\le2 .
\]
Equivalently, it is Stieltjes moment-indeterminate if and only if $1/p-1/q>2$. In particular, $V_{p,p}=1$ and the boundary case $1/p-1/q=2$ are determinate.
\end{theorem}

\begin{proof}
The case $p=q$ is immediate, since $V_{p,p}=1$. Assume henceforth that $p<q$, and put
\[
a=\frac1p,\qquad b=\frac1q,\qquad
\delta=a-b=\frac1p-\frac1q>0 .
\]
Let
\[
m_k=\E V_{p,q}^k,\qquad k=1,2,\ldots .
\]
By Lemma~\ref{lem:inverse-stable-size-bias},
\[
m_k
=
2^{\delta k}
\frac{\Gamma(a(k+1))\Gamma(b)}
{\Gamma(b(k+1))\Gamma(a)} .
\]
Stirling's formula gives, as $k\to\infty$,
\[
\log\Gamma(c(k+1))=ck\log k+O(k),
\qquad c>0 .
\]
Consequently,
\[
\log m_k
=
\delta k\log k+O(k),
\]
and hence
\[
m_k^{-1/(2k)}\asymp k^{-\delta/2}.
\]
Here $A_k\asymp B_k$ means that $A_k/B_k$ is bounded above and below by positive constants for all sufficiently large $k$.

By the Stieltjes version of Carleman's criterion, the divergence condition
\[
\sum_{k=1}^{\infty}m_k^{-1/(2k)}=\infty
\]
implies Stieltjes moment determinacy; see, for example, \citep{Stoyanov2000,Lin2017}. Since
\[
\sum_{k=1}^{\infty}m_k^{-1/(2k)}
\asymp
\sum_{k=1}^{\infty}k^{-\delta/2},
\]
the Carleman sum diverges exactly when
\[
\frac{\delta}{2}\le1,
\]
that is, when $\delta\le2$. Therefore $V_{p,q}$ is Stieltjes moment-determinate for $\delta\le2$.

It remains to prove moment indeterminacy for $\delta>2$. Let
\[
\alpha=\frac ba=\frac pq\in(0,1),
\]
and let $s_\alpha$ denote the density of the positive $\alpha$-stable random variable $S_\alpha$ normalized by
\[
\E e^{-uS_\alpha}=e^{-u^\alpha},\qquad u\ge0 .
\]
By Lemma~\ref{lem:inverse-stable-size-bias},
\[
V_{p,q}=2^\delta W,\qquad
W=(W_0)^*,\qquad
W_0=S_\alpha^{-b}.
\]
The density of $W_0$ is obtained by the change of variables $w=s^{-b}$:
\[
f_{W_0}(w)
=
\frac1b\,w^{-1/b-1}s_\alpha(w^{-1/b}),
\qquad w>0.
\]
Since $W$ is the $W_0$-size-biased version of $W_0$,
\[
f_W(w)
=
\frac{w f_{W_0}(w)}{\E W_0}
=
\frac{1}{b\,\E W_0}w^{-1/b}s_\alpha(w^{-1/b}),
\qquad w>0 .
\]
Therefore $V_{p,q}$ has density
\[
f_{p,q}(v)
=
\frac{1}{2^\delta}f_W(2^{-\delta}v)
=
\frac{1}{b\,2^\delta\,\E W_0}
(2^{-\delta}v)^{-1/b}
s_\alpha\!\left((2^{-\delta}v)^{-1/b}\right),
\qquad v>0 .
\]

We now use the classical small-argument asymptotic of the one-sided stable density. In the present normalization,
\[
-\log s_\alpha(x)
=
c_\alpha x^{-\alpha/(1-\alpha)}
+
O\!\left(\log\frac1x\right),
\qquad x\downarrow0,
\]
with $c_\alpha>0$; see Zolotarev \cite[Chapter~2]{Zolotarev1986} and Nolan \cite{Nolan2020}. Substituting
\[
x=(2^{-\delta}v)^{-1/b}
\]
gives
\[
-\log f_{p,q}(v)
=
C_{p,q}v^{\alpha/(b(1-\alpha))}
+
O(\log v),
\qquad v\to\infty,
\]
for some constant $C_{p,q}>0$. Since
\[
\frac{\alpha}{b(1-\alpha)}
=
\frac{b/a}{b(1-b/a)}
=
\frac1{a-b}
=
\frac1\delta,
\]
we obtain
\[
-\log f_{p,q}(v)
=
C_{p,q}v^{1/\delta}+O(\log v),
\qquad v\to\infty .
\]

By the Stieltjes Krein criterion, an absolutely continuous probability law on $(0,\infty)$ with positive density $f$ is Stieltjes moment-indeterminate if, for some $c>0$,
\[
\int_c^\infty
\frac{-\log f(x^2)}{1+x^2}\,\dd x
<\infty ;
\]
see \citep{Stoyanov2000,Lin2017}. Applying this criterion to $f=f_{p,q}$, the preceding logarithmic asymptotic gives
\[
-\log f_{p,q}(x^2)
=
C_{p,q}x^{2/\delta}+O(\log x),
\qquad x\to\infty .
\]
Thus the tail of the Krein integral has the same convergence behavior as
\[
\int^\infty \frac{x^{2/\delta}}{x^2}\,\dd x
=
\int^\infty x^{2/\delta-2}\,\dd x .
\]
This integral is finite if and only if
\[
2/\delta-2<-1,
\]
equivalently $\delta>2$. Hence $V_{p,q}$ is Stieltjes moment-indeterminate for $\delta>2$.

Combining the Carleman half and the Krein half proves the sharp criterion. At the boundary $\delta=2$, the Carleman sum is comparable to $\sum_k k^{-1}$, and therefore diverges; the boundary is determinate.
\end{proof}

\begin{corollary}[Bridge to the Gaussian-base factor of Dytso et al.]
\label{cor:dytso-gaussian-base-moment}
Let $0<r\le2$, and let $V_{G,r}$ denote the Gaussian-base factor of \citet{Dytso2018}, defined by
\[
X_r\stackrel{d}{=}V_{G,r}X_2,
\qquad
V_{G,r}>0,\qquad V_{G,r}\perp X_2 .
\]
Then
\[
V_{G,r}\stackrel{d}{=}V_{r,2},
\]
and
\[
V_{G,r}\ \text{is Stieltjes moment-determinate}
\quad\Longleftrightarrow\quad
r\ge \frac25 .
\]
Equivalently,
\[
V_{G,r}\ \text{is Stieltjes moment-indeterminate}
\quad\Longleftrightarrow\quad
0<r<\frac25 .
\]
Thus \citet[Proposition~8]{Dytso2018} is recovered and sharpened: their sufficient condition $r\ge2/5$ is the exact Gaussian-base threshold.
\end{corollary}

\begin{proof}
Taking $p=r$ and $q=2$ in Theorem~\ref{thm:main} gives
\[
X_r\stackrel{d}{=}V_{r,2}X_2 .
\]
The Mellin quotient in Lemma~\ref{lem:quotient-uniqueness} is unique, so the positive factor law agrees with the Gaussian-base factor law of \citet{Dytso2018}. Hence
\[
V_{G,r}\stackrel{d}{=}V_{r,2}.
\]
For $V_{r,2}$,
\[
\delta=\frac1r-\frac12 .
\]
By Theorem~\ref{thm:factor-moment-determinacy}, $V_{r,2}$ is moment-determinate if and only if
\[
\frac1r-\frac12\le2 .
\]
This is equivalent to
\[
\frac1r\le\frac52,
\]
or
\[
r\ge\frac25 .
\]
The indeterminate side is the complementary range $0<r<2/5$.
\end{proof}

The preceding corollary also explains the example emphasized by \citet{Dytso2018}. By their Proposition~2, the generalized Gaussian law $X_r$ is moment-indeterminate for $0<r<1$, while $X_2$ is moment-determinate. The corollary shows that $V_{G,r}$ is moment-determinate exactly for $r\ge2/5$. Hence, for
\[
\frac25\le r<1,
\]
the product representation
\[
X_r\stackrel{d}{=}V_{G,r}X_2
\]
expresses a moment-indeterminate random variable as the product of two independent moment-determinate random variables. Theorem~\ref{thm:factor-moment-determinacy} gives the full two-parameter version of this phenomenon: whenever
\[
p<1,\qquad q\ge1,\qquad \frac1p-\frac1q\le2,
\]
the variables $V_{p,q}$ and $X_q$ are moment-determinate, while
\[
X_p\stackrel{d}{=}V_{p,q}X_q
\]
is moment-indeterminate.

\begin{proposition}[Transport kernels]
\label{prop:transport}
For $0<p\le q$, define, for bounded measurable $h:\R\to\R$,
\[
K_{p,q}h(x)=\E h(V_{p,q}x),
\]
where $V_{p,q}$ has the unique factor law. Then $K_{p,q}$ is a Markov operator and
\[
\E h(X_p)=\E K_{p,q}h(X_q).
\]
Moreover, for $0<p\le q\le r$, the operators compose on test functions according to
\[
K_{q,r}(K_{p,q}h)=K_{p,r}h.
\]
The composition is evaluated using independent copies of the factor laws.
\end{proposition}

\begin{proof}
For bounded measurable $h$, the map $x\mapsto K_{p,q}h(x)$ is bounded and measurable by standard monotone-class approximation. Also $K_{p,q}1=1$, and $h\ge0$ implies $K_{p,q}h\ge0$. Hence $K_{p,q}$ is a Markov operator. The first identity is just $X_p\stackrel{d}{=}V_{p,q}X_q$, with $V_{p,q}\perp X_q$. For the composition law, take independent copies of $V_{p,q}$ and $V_{q,r}$. Corollary~\ref{cor:cocycle} gives $V_{p,q}V_{q,r}\stackrel{d}{=}V_{p,r}$, hence, for every $x\in\R$,
\[
K_{q,r}(K_{p,q}h)(x)=\E h(V_{p,q}V_{q,r}x)=K_{p,r}h(x).
\]
Thus the factors do not merely solve a Mellin-positivity problem; they define a consistent random-scaling calculus across the shape parameter.
\end{proof}

The kernel viewpoint gives an operational use for the theorem. When expectations, probabilities, or loss functionals are easier to evaluate under a tractable $q$-law, the identity above transfers them exactly to the heavier-tailed $p$-law, $p\le q$, by a positive random scale. This supports weighted importance estimation, simulation, and sensitivity analysis across tail shapes without numerical inversion of the Fourier--Mellin density. In particular, one can move from a convenient concentrated or Gaussian-like base law to a heavier-tailed target law by a canonical factor whose distribution is fixed by the pair $(p,q)$.

\begin{corollary}[Weighted stable representation]
\label{cor:weighted}
Let $0<p\le q$, and let $V_{p,q}$ be the unique factor of Theorem~\ref{thm:main}. Then, for every bounded measurable $h:\R\to\R$,
\[
\E h(X_p)=\E h(V_{p,q}X_q).
\]
If $p<q$, put $a=1/p$, $b=1/q$, let $S_{p/q}$ be independent of $X_q$, and set
\[
W_0=S_{p/q}^{-b}.
\]
Then
\[
\E W_0=\frac{\Gamma(a)}{(p/q)\Gamma(b)}=\frac{a\Gamma(a)}{b\Gamma(b)}<\infty,
\]
and
\begin{equation}
\E h(X_p)=\frac{\E\left[W_0 h\left(2^{a-b}W_0X_q\right)\right]}{\E W_0}.
\label{eq:weighted-identity}
\end{equation}
\end{corollary}

\begin{proof}
The first identity is just $X_p\stackrel{d}{=}V_{p,q}X_q$. For the second, $V_{p,q}=2^{a-b}W$, where $W$ is the $W_0$-size-biased version of $W_0$. Let
\[
g(w)=\E h(2^{a-b}wX_q),\qquad w>0,
\]
where $X_q$ is independent of $W_0$. Then $g$ is bounded and measurable, so the size-biasing identity gives
\[
\E h(V_{p,q}X_q)=\E g(W)=\frac{\E[W_0g(W_0)]}{\E W_0}.
\]
Expanding $g(W_0)$ and using the independence of $W_0$ and $X_q$ gives \eqref{eq:weighted-identity}.
\end{proof}

Formula \eqref{eq:weighted-identity} should be read as an exact weighted importance identity, not as a direct sampler for $V_{p,q}$. Direct unweighted sampling requires sampling the size-biased law of $W_0$; the displayed identity samples $W_0=S_{p/q}^{-1/q}$ and weights by $W_0$. One-sided positive stable variables can be generated by Kanter's formula \citep{Kanter1975}; Chambers--Mallows--Stuck routines require the matching one-sided parameterization \citep{Chambers1976,Nolan2020}. No numerical inversion of the Mellin density is required. The same identity extends to unbounded $h$ whenever the displayed expectations are finite; the bounded formulation above is used only to avoid additional integrability assumptions.

For bounded $h$, the weighted identity gives an unbiased estimator. Let $(W_i,Z_i)_{i\ge1}$ be independent copies of $(W_0,X_q)$. Then the sample average
\[
\widehat I_n(h)=\frac{1}{n\E W_0}\sum_{i=1}^n W_i h\left(2^{a-b}W_iZ_i\right)
\]
is unbiased for $\E h(X_p)$ and converges almost surely to that quantity. Moreover, for bounded $h$ it has finite variance, since
\[
\E W_0^2=\frac{\Gamma(2a)}{(p/q)\Gamma(2b)}<\infty
\]
by Lemma~2. In particular,
\[
\operatorname{Var}\widehat I_n(h)\le \frac{\|h\|_\infty^2}{n(\E W_0)^2}\E W_0^2.
\]
For unbounded $h$, the same statements require the corresponding weighted moments to be finite.

\begin{example}[A non-Gaussian branch beyond the Gaussian-mixture case]
Let $p=3$ and $q=4$. Then $a=1/3$, $b=1/4$, and $\alpha=3/4$. The theorem gives
\[
X_3\stackrel{d}{=}V_{3,4}X_4,
\qquad
V_{3,4}=2^{1/12}\left(S_{3/4}^{-1/4}\right)^*.
\]
This example lies entirely in the regime $p,q>2$, outside the Gaussian-mixture range highlighted by the classical $q=2$ base case. It is also outside the $0<p\le q\le2$ product subregion obtained from \citet[Proposition~5]{Dytso2018} under the reparameterization in the introduction.
\end{example}

\begin{example}[The Laplace law as a Gaussian scale mixture]
Let $p=1$ and $q=2$. Then $a=1$, $b=1/2$, and $\alpha=1/2$. The positive $1/2$-stable variable has density
\[
f_{S_{1/2}}(s)=\frac{1}{2\sqrt{\pi}}s^{-3/2}\exp\left(-\frac{1}{4s}\right),\qquad s>0.
\]
Thus $W_0=S_{1/2}^{-1/2}$ has density
\[
f_{W_0}(w)=\frac{1}{\sqrt{\pi}}e^{-w^2/4},\qquad w>0.
\]
Since $\E W_0=2/\sqrt{\pi}$, its size-biased version $W$ has density
\[
f_W(w)=\frac{w}{2}e^{-w^2/4},\qquad w>0.
\]
The factor $V=\sqrt{2}W$ therefore has density
\[
f_V(v)=\frac{v}{4}e^{-v^2/8},\qquad v>0.
\]
Consequently,
\[
X_1\stackrel{d}{=}VX_2.
\]
Equivalently, $V^2$ is exponential with rate $1/8$, so the conditional variance of $VX_2$ is exponentially mixed. The resulting density is $(1/4)e^{-|x|/2}$, exactly the $r=1$ member of \eqref{eq:density}. This is the classical Rayleigh-scale representation of the Laplace law, written in the present normalization.
\end{example}

\section{Conclusion}
\label{sec:conclusion}

We have proved an exact positive scale-mixture classification for the centered unit-scale generalized Gaussian family:
\[
X_p\stackrel{d}{=}VX_q,
\qquad V>0,\quad V\perp X_q,
\]
holds if and only if $p\le q$, and the law of the factor is unique. The nontrivial branch $p<q$ is constructive. The mixing factor is an explicit size-biased power of a positive stable random variable:
\[
V=2^{1/p-1/q}W,
\qquad
W=\left(S_{p/q}^{-1/q}\right)^*,
\]
where the star denotes ordinary size-biasing of the entire positive random variable $S_{p/q}^{-1/q}$. The proof identifies the Mellin quotient of the two absolute-value laws with the Mellin transform of this stable-size-biased factor, thereby proving positivity of the inverse Fourier--Mellin candidate throughout the $p<q$ branch, including the cases not already supplied by the earlier product subregion. The reverse direction $p>q$ is blocked because the necessary Mellin quotient for $\log V$ is unbounded. The uniqueness of the factor yields the distributional cocycle $V_{p,r}\stackrel{d}{=}V_{p,q}V_{q,r}$, so the result is an order theorem rather than a single isolated product identity.

Several extensions suggest themselves. First, one may ask for finer properties of the factor $V$, such as tail asymptotics, unimodality, infinite divisibility, or generalized gamma convolution membership \citep{Bondesson1992,Schilling2012}. Second, the result suggests multivariate analogues for radial or $p$-spherical generalized Gaussian distributions. Third, the factorization can be used as a transfer principle: expectations, weighted sampling schemes, and inequalities stable under the kernel $K_{p,q}$ can be transported from order $q$ to every lower order $p\le q$ by random scaling.

A further, more speculative direction is to connect the present exact scale factors with hierarchical nonequilibrium models in which observed distributions arise from fluctuating intensive or scale variables. Hierarchical models for complex systems and turbulence use stochastic cascade dynamics to generate broad stationary laws \citep{Macedo2017,SalazarVasconcelos2010}, while superstatistics models average local equilibrium distributions over fluctuating parameters \citep{Beck2005}. The result here should not be read as a physical derivation of those models. It suggests a narrower mathematical question: whether some hierarchical or superstatistical mechanisms produce, or are approximated by, the positive-stable size-biased scale laws that exactly transport one generalized Gaussian shape into another.

\section*{Acknowledgements}

The author acknowledges Dytso, Bustin, Poor and Shamai for posing the product-factorization question and for isolating the Mellin quotient that motivates the present construction.

\end{document}